\documentclass[12pt,reqno,b5paper]{amsart}
\usepackage{amsmath, amsthm}
\usepackage{latexsym}
\usepackage{epsfig}
\usepackage{amsfonts}
\usepackage{amssymb}
\usepackage{graphicx}
\usepackage{hyperref}

\newtheorem{thr}{Theorem}[section]

\newtheorem{lem}{Lemma}[section]

\newcommand{\F}{\mathbb{F}}
\newcommand{\K}{\mathbb{K}}
\newcommand{\Z}{\mathbb{Z}}

\begin{document}

%
%
%
%
%

\bigskip
\title[Swan-like reducibility for Type I pentanomials $\cdots$]{Swan-like reducibility for Type I pentanomials over a binary field}
\author[Ryul Kim, Su-Yong Pak and Myong-Son Sin]{Ryul Kim, Su-Yong Pak and Myong-Son Sin}
\maketitle

\textbf{Abstract.}
Swan (Pacific J. Math. 12(3) (1962), 1099-1106) characterized the parity of the number of irreducible factors of 
trinomials over $\F_2$. Many researchers have recently obtained Swan-like results 
on determining the reducibility of polynomials over finite fields. 
In this paper, we determine the parity of the number of irreducible factors for so-called 
Type I pentanomial $f(x)=x^m+x^{n+1}+x^n+x+1$ over $\F_2$ with even $n$. 
Our result is based on the Stickelberger-Swan theorem and Newton's formula which 
is very useful for the computation of the discriminant of a polynomial.\\

\textbf{(2010)Mathematics Subject Classification:} 11T06, 11T55, 12E05 \newline
%
%

\section{Introduction}
Irreducible polynomials over a finite field $\F_2$ with a small number of nonzero terms 
are important in many applications such as coding theory and cryptography 
because such polynomials provide with an efficient
implementation of field arithmetic in the field extension $\F_{2^m}$. But an irreducible trinomial
over $\F_2$ for every given degree does not always exist and a conjecture
whether there exists an irreducible pentanomial of degree $m$ over $\F_2$ for each
$m \geq 4$ still remains open.
	
On the other hand, characterization of the parity of the number of irreducible factors is meaningful for determining
the reducibility of a given polynomial. Since a polynomial is reducible if it has an even number of
irreducible factors, study on the parity of this number can give a necessary (but not sufficient) condition for irreducibility. Using a classical result of Stickelberger 
\cite{sti}, Swan \cite{swa} determined the parity of the number 
of irreducible factors of trinomials over $\F_2$. 
Swan's theorem relates the discriminant of a polynomial with its number of irreducible factors. 
	
Many Swan-like results have recently been obtained for several types of polynomials over finite fields. 
Vishne \cite{vis} extended the Swan's theorem to trinomials over an extension of $\F_2$. 
Hales and Newhart \cite{hal} gave a Swan-like result for binary tetranomials and 
Bluher \cite{blu} presented a similar result for binary polynomials of the 
form $x^n+\sum_{i \in S}x^i+1$, where 
$S \subset \{i : i ~ \text{odd}, 0<i<n/3 \} \bigcup \{i : i \equiv n \pmod{4}, 0<i<n \}$. 
Ahmadi and Menezes \cite{ahm1} obtained a Swan-like result for binary maximum weight polynomials.  
Zhao and Cao \cite{zha} considered the reducibility of binary affine polynomials. 
Some Swan-like results related to the reducibility of polynomials 
over finite fields of odd characteristic have been obtained, see \cite{ahm2, han, kim, loi, von}.
	
Types I and II pentanomials over $\F_2$ were firstly introduced in \cite{rod} as follows.
\begin{eqnarray*}
&& \textrm{Type I : } x^{m}+x^{n+1}+x^{n}+x+1, ~~~\textnormal{where}~~
	2 \leq n \leq \lfloor m/2 \rfloor -1 \\
&& \textrm{Type II : } x^{m}+x^{n+2}+x^{n+1}+x^{n}+1, ~~~ \textnormal{where} ~~\textrm{ }
	1 \leq n \leq \lfloor m/2 \rfloor -1 
\end{eqnarray*}
The authors proposed parallel multiplier architectures based on these special irreducible pentanomials 
and gave rigorous analyses of their space and time complexity.
Though these two types of irreducible pentanomials are abundant, they do not exist for each given degree. 
	
Koepf and Kim \cite{koe} determined the parity of the number of irreducible factors for 
Type II pentanomials over $\F_2$ with even degrees. In this work we consider the same problem for  
Type I pentanomials over $\F_2$ with even $n$ using the Stickelberger-Swan theorem and Newton's formula.
In Sect. 2 we present some preliminary results related to 
the parity of the number of irreducible factors of polynomials over finite fields. 
In Sect. 3 we determine the parity of the number of irreducible factors of Type I polynomials over 
$\F_2$ with even $n$ and In Sect. 4 we conclude.

%
%
%
%

\section{Preliminaries}
Let $\K$ be a field and let $f(x)= a \prod_{i=0}^{m-1}(x-x_i) \in \K[x]$, where $x_0, \cdots, x_{m-1}$
are the roots of $f(x)$ in an extension of $\K$. Then the \textit{discriminant} $D(f)$ of $f(x)$
is defined by
\begin{equation*}
D(f)=a^{2m-2} \prod_{0\leq i<j < m}(x_i-x_j)^2.
\end{equation*}
It is obvious from the definition of $D(f)$ that $f(x)$ has a repeated root if and only if $D(f)=0$. 
Although the discriminant is defined in terms of elements of an extension of $\K$, it is actually an element of $\K$ itself.
%
%
The following theorem, called the Stickelberger-Swan theorem, relates the parity of the number 
of irreducible factors of a polynomial with its discriminant.
\begin{thr} \textnormal{\cite{ahm2, swa}} \label{thr2.1}
Suppose that the polynomial $f(x) \in \F_2[x]$ of degree $m$
has no repeated roots and let $r$ be a number of irreducible factors of $f(x)$ over $\F_2$.
Let $F(x) \in \Z[x]$ be any monic lift of $f(x)$ to the integers. Then $D(F) \equiv 1$ or $5 \pmod{8}$, 
and more importantly, $r \equiv m \pmod{2}$ if and only if $D(F) \equiv 1 \pmod{8}$.
\end{thr}

Let $g(x)= b \prod_{j=0}^{n-1}(x-y_j) \in \K[x]$, where $y_0,\cdots,y_{n-1}$ are the roots of $g(x)$ in an extension of $\K$.
The \textit{resultant} $R(f, g)$  of $f(x)$ and $g(x)$ is
\begin{equation*}
R(f,g)=(-1)^{mn}b^m \prod_{j=0}^{n-1}f(y_j) = a^n \prod_{i=0}^{m-1}g(x_i).
\end{equation*}
There is an important relation between the discriminant and the resultant given by
\begin{equation*} 
D(f)=(-1)^{m(m-1)/2} R(f,f'),
\end{equation*}
where $f'(x)$ denotes the derivative of $f(x)$ with respect to $x$. This implies the following lemma.
%
%
\begin{lem} \textnormal{\cite{han}} \label{lem2.1}
An alternate formula for the discriminant of a monic polynomial $f(x)$ is
\begin{equation*}
D(f)=(-1)^{m(m-1)/2} \prod_{i=0}^{m-1}f'(x_i).
\end{equation*}
\end{lem}

Let
\begin{equation*} 
f(x)=x^m+a_1x^{m-1}+\cdots+a_m=\prod_{i=0}^{m-1}(x-x_i) \in \K[x].
\end{equation*}
It is well known that the coefficients $a_k$ of $f(x)$ are
the elementary symmetric polynomials of $x_i$ :
\begin{equation*} 
a_k=(-1)^k \sum_{0 \leq i_1<i_2\cdots<i_k<m} x_{i_1} x_{i_2} \cdots x_{i_k}
\end{equation*}
for $1\leq k<m$. Since each $a_k \in \K$, it follows that
$S(x_0, \cdots, x_{m-1}) \in \K$ for every symmetric polynomial
$S \in \K[x_0, \cdots, x_{m-1} ]$. For any integers $p,q$ and $k(0\leq k<m)$, let
\begin{eqnarray*} 
S_{(k,p)} & = & \sum_{\substack{0 \leq i_1,\cdots ,i_k \leq m-1 \\ i_j \neq i_l}}
	x_{i_1}^p \cdots x_{i_k}^p, \\
S_{[k,p]} & = & \sum_{0 \leq i_1<i_2<\cdots<i_k \leq m-1}
	x_{i_1}^p \cdots x_{i_k}^p, \\
S_{p,q} & = & \sum_{\substack{0 \leq i,j \leq m-1 \\ i \neq j}}
	x_i^p x_j^q
\end{eqnarray*}
We denote $S_{(1,p)}=S_{[1,p]}$ simply as $S_p$ and put $S_{(0,p)}=S_{[0,p]}=1$. 
Then the following lemma holds true.
%
%
\begin{lem} \textnormal{\cite{ahm1, koe}} \label{lem2.2}

\textnormal{(1)} $S_0 = S_{(1,0)}=S_{[1,0]}=m$

\textnormal{(2)} $S_{p,q}=S_p \cdot S_q-S_{p+q}$

\textnormal{(3)} $S_{(k,p)}=k! \cdot S_{[k,p]}$
\end{lem}
%
%

Newton's formula relates the coefficients $a_k$ with the power sums $S_k$.
\begin{thr} \textnormal{\cite{lid}} \label{thr2.2}
Let $f(x), S_p$ and $x_0,x_1,\cdots,x_{m-1}$ be as above. Then for any $p \geq 1$,
\begin{equation*} 
S_p+S_{p-1}a_1+S_{p-2}a_2+ \cdots +S_{p-n+1}a_{n-1}+\frac{n}{m} S_{p-n}a_{n}=0  
\end{equation*}
where $n=\textnormal{min} \{p,m\}$.
\end{thr}

The \textit{reciprocal polynomial} \cite{lid} of $f(x)=a_0x^m+a_1x^{m-1}+ \cdots +a_{m-1}x+a_m$ 
with $a_0 \neq 0$ over a finite field $\F_q$ is defined by
\begin{equation*} 
f^{*}(x):=x^mf \Big( \frac{1}{x} \Big)=a_mx^m+a_{m-1}x^{m-1}+ \cdots +a_1x+a_0 \in \F_q[x].
\end{equation*}

%
%
%
%
\section{The reducibility of Type I pentanomials over $\F_2$}
In this section, we will determine the parity of the number of irreducible factors
for the polynomial 
\begin{equation} 
f(x)=x^m+x^{n+1}+x^n+x+1 \in \F_2[x] \label{eq1}
\end{equation}
where $n$ is even and $2 \leq n \leq \lfloor m/2 \rfloor -1$.
First we test if these polynomials have repeated roots.
%
%
\begin{lem} \label{lem3.1}
The polynomial $f(x)$ in (1) has no repeated roots.
\end{lem}
\begin{proof}
The derivative of $f(x)$ is
\begin{eqnarray*}
f'(x) & = & mx^{m-1}+(n+1)x^n+nx^{n-1}+1 = mx^{m-1}+x^n+1 \\
	& = & \left\{ \begin{array}{ll}
		x^n+1, & m \equiv 0 \pmod{2} \\
		x^{m-1}+x^n+1, & m \equiv 1 \pmod{2}
	\end{array} \right.
\end{eqnarray*}
since $n$ is even. Therefore,
\begin{eqnarray*}
\textnormal{gcd}(f,f') & = & \left\{ \begin{array}{ll}
		\textnormal{gcd}(x^m, x^n+1), & m \equiv 0 \pmod{2} \\
		\textnormal{gcd}(x^n+1, x^{m-1}), & m \equiv 1 \pmod{2}
	\end{array} \right. \\
	& = & 1
\end{eqnarray*}
and $f(x)$ has no repeated roots. 
\end{proof}
\bigskip
Now let $F(x) \in \Z[x]$ be the monic lift of $f(x)$ to the integers. 
By Lemma 2.1, its discriminant is
\small
\begin{eqnarray*}
D(F) & = & (-1)^{m(m-1)/2}R(F,F') = (-1)^{m(m-1)/2} \prod_{i=0}^{m-1}F'(r_i) \\
	& = & (-1)^{m(m-1)/2} \prod_{i=0}^{m-1}\big(mr_i^{m-1}+(n+1)r_i^n+nr_i^{n-1}+1\big),
\end{eqnarray*}
\normalsize
where $r_i$'s are the roots of $F(x)$ in some extension of the rational numbers. To determine the 
parity of the number of irreducible factors of $f(x)$ by using Theorem 2.1, we need to compute $D(F)$ 
modulo $8$ which depends on the parity of the coefficients $m, n, n+1, 1$. It is not so desirable to 
compute $D(F)$ directly because at least two of these coefficients are odd.	

Let $h(x)$ be an arbitrary polynomial over $\F_2$. Then it is clear that $h^{*}(x)$, 
the reciprocal polynomial of $h(x)$ has the same degree as $h(x)$ and the same number of irreducible
factors as $h(x)$ since the reciprocal polynomial of an irreducible polynomial is also
irreducible. Let $p(x)$ be an irreducible polynomial of odd degree over $\F_2$ which divides
neither $h(x)$ nor $h^{*}(x)$. If we multiply $h(x)$ by $p(x)$, then both the parity of the number
of irreducible factors and the parity of the degree change. The same holds true for
$h^{*}(x)p(x)$. Using this fact, plus Stickelberger-Swan theorem, we
immediately get the following lemma.
%
%
\begin{lem} \label{lem3.2}
Let $h(x)$ be a polynomial over $\F_2$ which has not repeated roots and $p(x)$ be an irreducible 
polynomial of odd degree over $\F_2$ which divides neither $h(x)$ nor $h^*(x)$. 
Then the discriminants modulo 8 of monic lifts of the polynomials $h(x), h^*(x), p(x)h(x), p(x)h^*(x)$ are same.
\end{lem}
%
%
%
%
\subsection{The case of odd degree}
Let the degree $m$ of $f(x)$ in (1) be odd and assume $n\geq 4$. The case of $n=2$ will be considered later. 
The monic lift of the polynomial $(x+1)f^{*}(x)$ to the integers is
\begin{equation*}
K(x):=x^{m+1}+x^{m-1}+x^{m-n+1}+x^{m-n-1}+x+1 \in \Z[x].\\
\end{equation*}
If $x_i$'s are the roots of $K(x)$ in some extension of the rational numbers, then
\begin{eqnarray} 
R(K,K') & = & \prod_{i=0}^{m} [(m+1)x_i^m+(m-1)x_i^{m-2}+(m-n+1)x_i^{m-n} \nonumber \\
	& & +(m-n-1)x_i^{m-n-2}+1]. \label{eq2}
\end{eqnarray}
By Theorem 2.1 and Lemma 3.2, $D(F) \equiv 1 \pmod{8}$ if and only if $D(K) \equiv 1 \pmod{8}$ since
$f(1) \neq 0$. So it suffices to compute $R(K,K')$ modulo 8. From (2), We have
%
\begin{align*}
R(K,K') \equiv &1+(m+1)\sum_{i=0}^{m}x_i^m+(m-1)\sum_{i=0}^{m}x_i^{m-2} \\
& +(m-n+1)\sum_{i=0}^{m}x_i^{m-n}+(m-n-1)\sum_{i=0}^{m}x_i^{m-n-2} \\
& +(m+1)^2\sum_{i<j}x_i^mx_j^m+(m-1)^2\sum_{i<j}x_i^{m-2}x_j^{m-2} \\
& +(m-n+1)^2\sum_{i<j}x_i^{m-n}x_j^{m-n} \\
& +(m-n-1)^2\sum_{i<j}x_i^{m-n-2}x_j^{m-n-2} \\
& +(m+1)(m-n+1)\sum_{i \neq j}x_i^mx_j^{m-n} \\
& +(m+1)(m-n-1)\sum_{i \neq j}x_i^mx_j^{m-n-2} \\
& +(m-1)(m-n+1)\sum_{i \neq j}x_i^{m-2}x_j^{m-n} 
\end{align*}
\begin{align*}
& \quad +(m-1)(m-n-1)\sum_{i \neq j}x_i^{m-2}x_j^{m-n-2} \pmod{8}.
\end{align*}
We adopt the symbols $S_k$ and $S_{p,q}$ in Sect. 2. Note that the number of roots is $m+1$ in this case.
Applying Lemma 2.2, we have
%
\begin{align}
&R(K,K') \equiv 1+(m+1)S_m+(m-1)S_{m-2}+(m-n+1)S_{m-n}\nonumber\\ 
& \quad +(m-n-1)S_{m-n-2}+\frac{1}{2}(m+1)^2S_{m,m}\nonumber \\
& \quad +\frac{1}{2}(m-1)^2S_{m-2,m-2}+\frac{1}{2}(m-n+1)^2S_{m-n,m-n}\nonumber\\
& \quad +\frac{1}{2}(m-n-1)^2S_{m-n-2,m-n-2}+(m+1)(m-n+1)S_{m,m-n}\nonumber\\
& \quad +(m+1)(m-n-1)S_{m,m-n+2}+(m-1)(m-n+1)S_{m-2,m-n}\nonumber\\
& \quad +(m-1)(m-n-1)S_{m-2,m-n-2}\nonumber\\
& \equiv  1+(m+1)S_m+(m-1)S_{m-2}+(m-n+1)S_{m-n} \nonumber\\
& \quad +(m-n-1)S_{m-n-2}+\frac{1}{2}(m+1)^2S_m^2 \nonumber\\
& \quad -\frac{1}{2}(m+1)^2S_{2m}+\frac{1}{2}(m-1)^2S_{m-2}^2 \\
& \quad -\frac{1}{2}(m-1)^2S_{2m-4}+\frac{1}{2}(m-n+1)^2S_{m-n}^2 \nonumber\\
& \quad -\frac{1}{2}(m-n+1)^2S_{2m-2n}+\frac{1}{2}(m-n-1)^2S_{m-n-2}^2 \nonumber \\
%
& \quad -\frac{1}{2}(m-n-1)^2S_{2m-2n-4}+(m+1)(m-n+1)S_mS_{m-n} \nonumber\\
& \quad -(m+1)(m-n+1)S_{2m-n}+(m+1)(m-n-1)S_mS_{m-n-2} \nonumber\\
& \quad -(m+1)(m-n-1)S_{2m-n-2}+(m-1)(m-n+1)S_{m-2}S_{m-n} \nonumber\\
& \quad -(m-1)(m-n+1)S_{2m-n-2}+(m-1)(m-n-1)S_{m-2}S_{m-n-2}  \nonumber \\
& \quad -(m-1)(m-n-1)S_{2m-n-4} \pmod{8}.\nonumber
\end{align}
%
All coefficients of the polynomial $K(x)$ are zero except
$a_0=a_2=a_n=a_{n+2}=a_m=a_{m+1}=1$, where $a_k$ is the
coefficient of the term with degree $m-k+1$.
In order to compute each term of (3), we apply Theorem 2.2 to get
\begin{align}
&S_0 = m+1 \nonumber\\
&S_1 = -a_1 = 0 \nonumber\\
&S_2 = -(S_1a_1+2a_2) = -2 \nonumber\\
&S_k+S_{k-2} = 0, ~ 2<k<n \nonumber\\
&S_n+S_{n-2} = -n \nonumber\\
&S_{n+1}+S_{n-1} = 0 \\
&S_{n+2}+S_n = -n \nonumber\\
&S_k+S_{k-2}+S_{k-n}+S_{k-n-2} = 0, ~ n+2<k<m \nonumber\\
&S_m+S_{m-2}+S_{m-n}+S_{m-n-2} = -m \nonumber\\
&S_{m+1}+S_{m-1}+S_{m-n+1}+S_{m-n-1} = -m-1 \nonumber\\
&S_k+S_{k-2}+S_{k-n}+S_{k-n-2}+S_{k-m}+S_{k-m-1} = 0, k>m+1.\nonumber
\end{align}
%
From (4), we see that $S_m=-m$ and $S_k=0$ for all odd number $k<m$, which is
followed by $S_{m-2}=S_{m-n}=S_{m-n-2}=0$. If we observe $S_k$ modulo 4
for even number $k$, then
\begin{align*}
&S_0 \equiv  m+1 \pmod{4} \\
&S_2 \equiv 2 \pmod{4} \\
&S_k+S_{k-2} \equiv 0 \pmod{4}, 2<k<n, 2 \mid k \\
&S_n+S_{n-2} \equiv n \pmod{4} \\
&S_{n+2}+S_n \equiv n \pmod{4} \\
&S_k+S_{k-2} \equiv S_{k-n}+S_{k-n-2} \pmod{4}, n+2<k<m, 2 \mid k \\
&S_k+S_{k-2} \equiv S_{k-n}+S_{k-n-2}+S_{k-m-1} \pmod{4}, m<k<2m, 2 \mid k \\
&S_{2m}+S_{2m-2} \equiv S_{2m-n}+S_{2m-n-2}-S_m+S_{m-1} \pmod{4},
\end{align*}
which shows that $S_{2m}$ is odd and $S_k$ is even for each
even number $k<2m$. Therefore, (3) can be simplified
as follows:
\begin{eqnarray}
R(K,K') & \equiv & 1-m(m+1)+\frac{1}{2}m^2(m+1)^2-\frac{1}{2}(m+1)^2S_{2m} \nonumber\\
	& & -\frac{1}{2}(m-1)^2S_{2m-4}-\frac{1}{2}(m-n+1)^2S_{2m-2n} \\
	& & -\frac{1}{2}(m-n-1)^2S_{2m-2n-4} \pmod{8}.\nonumber
\end{eqnarray}  \par
\bigskip
\noindent \textbf{CASE 1:} $n \equiv 0 \pmod{4}$ and $m+1 \equiv 0 \pmod{4}$ \par
\bigskip
In (5), $-\frac{1}{2}(m+1)^2S_{2m}$ and $-\frac{1}{2}(m-n+1)^2S_{2m-2n}$ become extinct, 
so we need to compute only $-\frac{1}{2}(m-1)^2S_{2m-4}$ and $-\frac{1}{2}(m-n-1)^2S_{2m-2n-4}$ modulo 8. 
Since $S_k \equiv 2 \pmod{4}$ for each even number $k(0<k<m)$, we get
\begin{equation*}
S_{m+k}+S_{m+k-2} \equiv S_{m+k-n}+S_{m+k-n-2}+2 \pmod{4}
\end{equation*}
for each odd number $k(1<k<m)$. Thus
\begin{eqnarray*}
S_{2m-4}+S_{2m-n-4} &\equiv& S_{2m-6}+S_{2m-n-6}+2 \\
	 & \equiv& \cdots \equiv S_{m-1}+S_{m-n+1}+2 \cdot \frac{m-3}{2} \pmod{4}.
\end{eqnarray*}
In similar way, we can easily check that 
\begin{equation*}
S_{2m-n-4}+S_{2m-2n-4} \equiv S_{m-1}+S_{m-n-1}+2 \cdot \frac{m-n-3}{2} \pmod{4}.
\end{equation*}
Therefore
\begin{equation*}
S_{2m-4}+S_{2m-2n-4} \equiv 0 \pmod{4}
\end{equation*}
and
\begin{eqnarray*}
D(K) & \equiv & (-1)^{m(m+1)/2}R(K,K') \\
	& \equiv & 1-m(m+1)-2(S_{2m-4}+S_{2m-2n-4}) \\
	& \equiv &1-m(m+1) \pmod{8}.
\end{eqnarray*} \par
\bigskip
\noindent \textbf{CASE 2:} $n \equiv 0 \pmod{4}$ and $m+1 \equiv 2 \pmod{4}$ \par
\bigskip
Similarly, we need to compute only $-\frac{1}{2}(m+1)^2S_{2m}$ 
and $-\frac{1}{2}(m-n+1)^2S_{2m-2n}$ in (5) modulo 8. We have
\begin{equation*}
S_{2m}+S_{2m-n} \equiv S_{m-1}+S_{m-n-1}-S_m+2 \cdot \frac{m+1}{2} \pmod{4}
\end{equation*}
and 
\begin{equation*}
S_{2m-n}+S_{2m-2n} \equiv S_{m-1}+S_{m-n-1}+2 \cdot \frac{m-n+1}{2} \pmod{4}.
\end{equation*}
Thus we have
\begin{equation*}
S_{2m}+S_{2m-2n} \equiv -S_m+n \equiv -S_m \equiv 1 \pmod{4},
\end{equation*}
which is followed by $D(K) \equiv 1+m(m+1)-2m^2 \pmod{8}$. \par
\bigskip
\noindent \textbf{CASE 3:} $n \equiv 2 \pmod{4}$ and $m+1 \equiv 0 \pmod{4}$ \par
\bigskip
\indent Assume that $k$ is an even number with $k<m$. Then we see easily
\begin{equation*}
S_k \equiv \left\{ \begin{array}{ll}
	0, & k \equiv 0 \pmod{n}\\
	2, & otherwise
	\end{array} \right.
\pmod{4}
\end{equation*}
and
\begin{equation*}
S_k+S_{k-2} \equiv \left\{ \begin{array}{ll}
	2, & k \equiv 0, 2 \pmod{n} \\
	0, & otherwise
	\end{array} \right.
\pmod{4}.
\end{equation*}
From (4), we have
\begin{align*}
&S_{m+1+ln}+S_{m-1+ln} \equiv S_{m-n+1}+S_{m-n-1} \pmod{4} \\
&S_{m+3+ln}+S_{m+1+ln} \equiv 2(l+1)+S_{m-n+3}+S_{m-n+1} \pmod{4} \\
& \qquad  \cdots	\qquad  \cdots \\
&S_{m-1+n+ln}+S_{m-3+n+ln} \equiv 2(l+1)+S_{m-1}+S_{m-3} \pmod{4}.
\end{align*}
Adding all these equations we get
\begin{equation*}
S_{m-1+ln}+S_{m-1+(l+1)n} \equiv S_{m-1}+S_{m-n-1} \pmod{4}
\end{equation*}
and thus
\begin{equation*}
S_k+S_{k-2n} \equiv 0 \pmod{4}
\end{equation*}
for each even number $k$ with $m+1<k<2m$. Therefore
\begin{eqnarray*}
&&S_{2m-2n}+S_{2m-4} \equiv S_{2m-2n}+S_{2m-2n-4} \\
	 && \quad \equiv  S_{2m-2n}+S_{2m-2n-2}+S_{2m-2n-2}+S_{2m-2n-4} \pmod{4},
\end{eqnarray*}
which depends on the residue 
$2m ~\textnormal{mod} ~n$. Now one can easily check that if $n \neq 6$, then
\begin{equation*}
S_{2m-4}+S_{2m-2n} \equiv \left\{ \begin{array}{ll}
	2, & 2m ~\textnormal{mod} ~n \leq 6 \\
	0, & 2m ~\textnormal{mod} ~n>6
\end{array} \right. \pmod{4}
\end{equation*}
and if $n=6$, then
\begin{equation*}
S_{2m-4}+S_{2m-2n} \equiv \left\{ \begin{array}{ll}
	2, & 2m ~\textnormal{mod} ~n > 0 \\
	0, & 2m ~\textnormal{mod} ~n = 0
\end{array} \right. \pmod{4}.
\end{equation*}
Thus the discriminant 
\begin{equation*}
D(K) \equiv 1-m(m+1)-2(S_{2m-4}+S_{2m-2n}) \pmod{8}
\end{equation*}
can be computed. \par
\bigskip
\textbf{CASE 4:} $n \equiv 2 \pmod{4}$ and $m+1 \equiv 2 \pmod{4}$ \par
\bigskip
\indent In similar way to above case, we see that if $n \neq 6$, then
\begin{equation*}
S_{2m}+S_{2m-2n-4} \equiv \left\{ \begin{array}{ll}
	3, & 2m ~\textnormal{mod} ~n \leq 6 \\
	1, & 2m ~\textnormal{mod} ~n>6
\end{array} \right. \pmod{4}
\end{equation*}
and if $n=6$, then
\begin{equation*}
S_{2m}+S_{2m-2n-4} \equiv \left\{ \begin{array}{ll}
	3, & 2m ~\textnormal{mod} ~n > 0 \\
	1, & 2m ~\textnormal{mod} ~n = 0
\end{array} \right. \pmod{4}.
\end{equation*}
And the discriminant is
\begin{equation*}
D(K) \equiv -1+m(m+1)-2m^2+2(S_{2m}+S_{2m-2n-4}) \pmod{8}.
\end{equation*}	
	
Summarizing the above argument, we obtain the following theorem.
\begin{thr}
Suppose $m$ is odd and $n>2$. Then the pentanomial $f(x)$ in (1) 
has an even number of irreducible factors over $\F_2$ if and
only if one of the following conditions holds :\\
\indent \textnormal{(1)} $n \equiv 0 \pmod{4}$ and $m \equiv \pm 3 \pmod{8}$; \\
\indent \textnormal{(2)} $n \neq 6$, $n \equiv 2 \pmod{4}$ and \textnormal{(a)}
	$2m ~\textnormal{mod} ~n \leq 6, m \equiv \pm 1 \pmod{8}$, or
	\textnormal{(b)} $2m ~\textnormal{mod} ~n>6, m \equiv \pm 3 \pmod{8}$; \\
\indent \textnormal{(3)} $n=6$ and \textnormal{(c)} $2m ~\textnormal{mod} ~n=0, m \equiv \pm 3 \pmod{8}$, or
	\textnormal{(d)} $2m ~\textnormal{mod} ~n \neq 0, m \equiv \pm 1 \pmod{8}$.
\end{thr}
%
%
%
%
\subsection{The case of even degree}

Now let the degree $m$ of $f(x)$ in (1) be even and assume $n\geq 4$. 
Observe the monic lift of the polynomial $(x+1)f(x)$ to the integers
\begin{equation*}
L(x):=x^{m+1}+x^m+x^{n+2}+x^n+x^2+1 \in \Z[x].\\
\end{equation*}
Let $x_i$'s be the roots of $L(x)$ in some extension of the rational numbers and $S_k$'s be defined 
similarly as in Sect. 2. Then we have the following.
\begin{align*}
&R(L,L')  \equiv (m+1)^{m+1}+(m+1)^m \left[ m\sum_{i=0}^{m}x_i^{-1}+(n+2) \sum_{i=0}^{m}x_i^{n+1-m} \right.\\
	& \left. \quad +n\sum_{i=0}^{m}x_i^{n-1-m}+2\sum_{i=0}^{m}x_i^{1-m} \right]+ \frac{1}{2}(m+1)^{m-1} \left[m^2\sum_{i<j}x_i^{-1}x_j^{-1} \right.\\
	& \quad +(n+2)^2\sum_{i<j}x_i^{n+1-m}x_j^{n+1-m}+n^2\sum_{i<j}x_i^{n-1-m}x_j^{n-1-m} \\
	& \left. \quad +2^2\sum_{i<j}x_i^{1-m}x_j^{1-m} \right] +(m+1)^{m-1} \left[ m(n+2)\sum_{i \neq j}x_i^{-1}x_j^{n+1-m} \right. \\
	& \quad +mn\sum_{i \neq j}x_i^{-1}x_j^{n-1-m} + 2m\sum_{i \neq j}x_i^{-1}x_j^{1-m}+2(n+2)\sum_{i \neq j}x_i^{n+1-m}x_j^{1-m} \\
	& \left. \quad +2n\sum_{i \neq j}x_i^{1-m}x_j^{n-1-m}\right]
\end{align*}

\begin{align*}
	& \equiv (m+1)+[mS_{-1}+(n+2)S_{n+1-m}+nS_{n-1-m}+2S_{1-m}] \nonumber\\
	& \quad +\frac{1}{2}(m+1)[m^2S_{-1}^2+(n+2)^2S_{n+1-m}^2+n^2S_{n-1-m}^2+4S_{1-m}^2] \nonumber\\
	&\quad -\frac{1}{2}(m+1)[m^2S_{-2}+(n+2)^2S_{2n+2-2m}+n^2S_{2n-2-2m}+4S_{2-2m}] \nonumber\\
	&\quad +[m(n+2)S_{-1}S_{n+1-m}+mnS_{-1}S_{n-1-m}+2mS_{-1}S_{1-m} \nonumber\\
	&\quad +2(n+2)S_{1-m}S_{n-1-m}+2nS_{n-1-m}S_{1-m}] \nonumber\\
	&\quad -[m(n+2)S_{n-m}+mnS_{n-2-m}+2mS_{-m}+n(n+2)S_{2n-2m}+ \nonumber\\
	&\quad +2(n+2)S_{n+2-2m}+2nS_{n-2m}] \pmod{8}.
\end{align*}

%
\noindent Let $T_k:=\sum_{i=0}^{m}y_i^k$ where $y_i$'s are the roots of $L^{*}(x)$,
the reciprocal of $L(x)$  in some extension of the rational numbers. 
Then clearly $S_{-k}=T_k$. By Theorem 2.2, we get the equations for $T_k$ same as (4).
So $T_k=0$ for each odd number $k$ with $1 \leq k<m+1$ and $T_{m+1}=-m-1$.
Meanwhile, for each even number $k$ with $1 \leq k<2m$, $T_k$ is also even. \par
\bigskip
\noindent \textbf{CASE 1:} $n \equiv 0 \pmod{4}$ \par
\bigskip
\indent In this case, we see easily that $T_{2m-2n-2} \equiv T_{2m-2} \pmod{4}$.
So $T_{2m-2n-2}+T_{2m-2} \equiv 0 \pmod{4}$ since $T_{2m-2n-2}$ and
$T_{2m-2}$ are all even. Therefore, we have
$D(L) \equiv (-1)^{m(m+1)/2}(m+1)(1+m^2) \pmod{8}$. \par
\bigskip
\noindent \textbf{CASE 2:} $n \equiv 2 \pmod{4}$ \par
\bigskip
\indent It is not difficult to check in similar way that if $n \neq 6$, then
\begin{equation*}
T_{2m-2n+2}+T_{2m-2n-2} \equiv \left\{ \begin{array}{ll}
	2, & 2m+4 ~\textnormal{mod} ~n \leq 6 \\
	0, & 2m+4 ~\textnormal{mod} ~n>6
\end{array} \right. \pmod{4}
\end{equation*}
and if $n=6$, then
\begin{equation*}
T_{2m-2n+2}+T_{2m-2n-2} \equiv \left\{ \begin{array}{ll}
	2, & 2m+4 ~\textnormal{mod} ~n > 0 \\
	0, & 2m+4 ~\textnormal{mod} ~n = 0
\end{array} \right. \pmod{4}.
\end{equation*}
Thus the discriminant
\begin{equation*}
D(L) \equiv (-1)^{m(m+1)/2} (m+1)[1+m^2-2(T_{2m-2n+2}+T_{2m-2n-2})] \pmod{8}
\end{equation*}
can be computed. \par
Summarizing the above consideration implies the following theorem.
\begin{thr}
Suppose $m$ is even and $n>2$. Then $f(x)$ in (1) 
has an even number of irreducible factors over $\F_2$ if and
only if one of the following conditions holds :\\
\indent \textnormal{(1)} $n \equiv 0 \pmod{4}$ and $m \equiv 0,2 \pmod{8}$; \\
\indent \textnormal{(2)} $n \neq 6$, $n \equiv 2 \pmod{4}$ and \textnormal{(a)}
	$2m+4 ~\textnormal{mod} ~n \leq 6, m \equiv 4, 6 \pmod{8}$, or
	\textnormal{(b)} $2m+4 ~\textnormal{mod} ~n>6, m \equiv 0, 2 \pmod{8}$; \\
\indent \textnormal{(3)} $n=6$ and \textnormal{(c)} $2m+4 ~\textnormal{mod} ~n=0, m \equiv 0, 2 \pmod{8}$, or
	\textnormal{(d)} $2m+4 ~\textnormal{mod} ~n \neq 0, m \equiv 4, 6 \pmod{8}$. \\
\end{thr}
%
%
%
%
\subsection{The case of $n=2$} 

In this subsection we consider the parity of the number of irreducible factors for the pentanomial
\begin{equation}
f(x)=x^m+x^3+x^2+x+1
\end{equation}
over $\F_2$. First assume that $m$ is odd. Then the monic lift of the polynomial $(x+1)f^*(x)$ to the integers is
\begin{equation*}
K=x^{m+1}+x^{m-3}+x+1
\end{equation*}
and 
\small
\begin{eqnarray*}
R(K,K') &\equiv &1+(m+1)S_m+(m-3)S_{m-4}\\
	& & +\frac{1}{2}(m+1)^2(S_m^2-S_{2m})+\frac{1}{2}(m-3)^2(S_{m-4}^2-S_{2m-8})\\
	& & +(m+1)(m-3)(S_mS_{m-4}-S_{2m-4}) \pmod{8}.
\end{eqnarray*}
\normalsize
Compute $S_k$'s using Lemma 2.2 and Theorem 2.2. Then we have
\small
\begin{eqnarray*}
R(K,K') &\equiv & 1-m(m+1)+\frac{1}{2}m^2(m+1)^2-\frac{1}{2}(m+1)^2S_{2m}\\
	& & -\frac{1}{2}(m-3)^2S_{2m-8}-(m+1)(m-3)S_{2m-4} \pmod{8}
\end{eqnarray*}
\normalsize
and therefore 
\begin{eqnarray}
D(K) & \equiv & \left\{ \begin{array}{ll}
		5\pmod{8}, & m \equiv \pm 3 \pmod{8} \\
		1\pmod{8}, & m \equiv \pm 1 \pmod{8}.
	\end{array} \right. 
\end{eqnarray}

Next assume that $m$ is even. Then the monic lift of the polynomial $(x+1)f(x)$ to the integers is
\begin{equation*}
L=x^{m+1}+x^m+x^4+1
\end{equation*}
and similarly we have
\begin{eqnarray}
D(L) & \equiv & \left\{ \begin{array}{ll}
		5\pmod{8}, & m \equiv 2, 4 \pmod{8} \\
		1\pmod{8}, & m \equiv 0, 6 \pmod{8}.
	\end{array} \right. 
\end{eqnarray}
By (7), (8) and Lemma 3.2, we have the following theorem.
\begin{thr}
The polynomial $f(x)$ in (6) has an even number of irreducible factors over $\F_2$ if and
only if $m \equiv 0, 3, 5, 6 \pmod{8}$.
\end{thr}


%
%
\section{Conclusion}
We have completely determined the parity of the number of irreducible factors for Type I
pentanomials (1) when $n$ is even. Our discussion is based on the 
Stickelberger-Swan theorem and somewhat complicated computation. In \cite{koe}, 
Type II pentanomials of even degrees were studied.
The results for Type II pentanomials of odd degrees with $n>2$ and Type I pentanomials
with odd $n$ still remain open.\\
	
{\bf Acknowledgement}. We would like to thank anonymous referees for their valuable comments and suggestions. 
%
%

\noindent Faculty of Mathematics, \textbf{Kim Il Sung} University, Kumsong Street, \\
Pyongyang, Democratic People's Republic of Korea, \\
e-mail address: ryul\_kim@yahoo.com \\
e-mail address: paksuyong@yahoo.com \\
e-mail address: sinmyongson@yahoo.com 


\begin{thebibliography}{20}
	
\bibitem{ahm1}
O. Ahmadi and A. Menezes, \textbf{Irreducible polynomials over maximum weight}, Util. Math. 72 (2007), 111-123.

\bibitem{ahm2}
O. Ahmadi and G. Vega, \textbf{On the parity of the number of irreducible factors of
self-reciprocal polynomials over finite fields}, Finite Fields Appl. 14(1) (2008), 124-131.

\bibitem{blu}
A. Bluher, \textbf{A Swan-like theorem}, Finite Fields Appl. 12(1) (2006), 128-138.

\bibitem{hal}
A. Hales A and D. Newhart, \textbf{Swan's theorem for binary tetranomials}, Finite Fields Appl. 12(2) (2006), 301-311.

\bibitem{han}
B. Hanson, D. Panario and D. Thomson, \textbf{Swan-like results for binomials and trinomials 
over finite fields of odd characteristic}, Des. Codes Cryptogr. 61(3) (2011), 273-283.

\bibitem{kim}
R. Kim and W. Koepf, \textbf{Parity of the number of irreducible factors for composite polynomials}, 
Finite Fields Appl. 16(3) (2010), 137-143.
	
\bibitem{koe}
W. Koepf and R. Kim, \textbf{The parity of the number of irreducible factors for some pentanomials}, 
Finite Fields Appl. 15(5) (2009), 585-603.

\bibitem{lid}
R. Lidl and H. Niederreiter, \textbf{Introduction to finite fields and their applications}, 
Cambridge University Press, 1997.
	
\bibitem{loi}
P. Loidreau, \textbf{On the factorization of trinomials over $\F_3$}, INRIA rapport de recherche, no. 3918, 2000.

\bibitem{rod}
F. Rodr\'iguez-Henr\'iquez and \c C.K. Ko\c c, \textbf{Parallel multipliers based on special irreducible
pentanomials}, IEEE Trans. Comput. 52(12) (2003), 1535-1542.

\bibitem{sti}
L. Stickelberger, \textbf{\"Uber eine neue Eigenschaft der Diskriminanten algebraischer Zahlk\"orper}, 
Verhandlungen des ersten Internationalen Mathematiker-Kongresses, 1897, Z\"urich, 182-193.
	
\bibitem{swa}
R. Swan, \textbf{Factorization of polynomials over finite fields}, Pacific J. Math. 12(3) (1962), 1099-1106.

\bibitem{vis}
U. Vishne, \textbf{Factorization of trinomials over Galois fields of characteristic 2}, 
Finite Fields Appl. 3(4) (1997), 370-377.

\bibitem{von}
J. von zur Gathen, \textbf{Irreducible trinomials over finite fields}. Math. Comp. 72(244) (2003), 1987-2000.

\bibitem{zha}
Z. Zhao and X. Cao, \textbf{A note on the reducibility of binary affine polynomials}, 
Des. Codes Cryptogr. 57(1) (2010), 83-90. \\

\end{thebibliography}
\end{document}